\documentclass[11pt,reqno]{amsart}

\newtheorem{proposition}{Proposition}[section]
\newtheorem{lemma}[proposition]{Lemma}
\newtheorem{corollary}[proposition]{Corollary}
\newtheorem{theorem}[proposition]{Theorem}

\theoremstyle{definition}
\newtheorem{definition}[proposition]{Definition}
\newtheorem{example}[proposition]{Example}

\newtheorem{remark}[proposition]{Remark}
\newtheorem{remarks}[proposition]{Remarks}

\newcommand{\thlabel}[1]{\label{th:#1}}
\newcommand{\thref}[1]{Theorem~\ref{th:#1}}
\newcommand{\selabel}[1]{\label{se:#1}}
\newcommand{\seref}[1]{Section~\ref{se:#1}}
\newcommand{\lelabel}[1]{\label{le:#1}}
\newcommand{\leref}[1]{Lemma~\ref{le:#1}}
\newcommand{\prlabel}[1]{\label{pr:#1}}
\newcommand{\prref}[1]{Proposition~\ref{pr:#1}}
\newcommand{\colabel}[1]{\label{co:#1}}
\newcommand{\coref}[1]{Corollary~\ref{co:#1}}
\newcommand{\relabel}[1]{\label{re:#1}}
\newcommand{\reref}[1]{Remark~\ref{re:#1}}
\newcommand{\exlabel}[1]{\label{ex:#1}}
\newcommand{\exref}[1]{Example~\ref{ex:#1}}
\newcommand{\delabel}[1]{\label{de:#1}}
\newcommand{\deref}[1]{Definition~\ref{de:#1}}
\newcommand{\eqlabel}[1]{\label{eq:#1}}
\newcommand{\equref}[1]{(\ref{eq:#1})}

\def\RR{{\mathbb R}}

\newcommand{\Cc}{\mathcal{C}}

\def\*C{{}^*\hspace*{-1pt}{\Cc}}
\def\text#1{{\rm {\rm #1}}}

\input xy
\xyoption {all} \CompileMatrices

\usepackage{amssymb}
\usepackage{color,amssymb,graphicx,amscd,amsmath,dsfont,stmaryrd,xcolor,comment}
\usepackage[pagebackref,colorlinks,urlcolor=blue,linkcolor=blue,citecolor=blue]{hyperref}

\begin{document}

\title[Unified products for Jordan algebras. Applications]
{Unified products for Jordan algebras. Applications}

\author{A. L. Agore}
\address{Simion Stoilow Institute of Mathematics of the Romanian Academy, P.O. Box 1-764, 014700 Bucharest, Romania and Vrije Universiteit Brussel, Pleinlaan 2, B-1050 Brussels, Belgium}
%\address{Vrije Universiteit Brussel, Pleinlaan 2, B-1050 Brussels, Belgium}
\email{ana.agore@vub.be and ana.agore@gmail.com}

\author{G. Militaru}
\address{Faculty of Mathematics and Computer Science, University of Bucharest, Str.
Academiei 14, RO-010014 Bucharest 1, Romania and Simion Stoilow Institute of Mathematics of the Romanian Academy, P.O. Box 1-764, 014700 Bucharest, Romania}
%\address{Simion Stoilow Institute of Mathematics of the Romanian Academy, P.O. Box 1-764, 014700 Bucharest, Romania}
\email{gigel.militaru@fmi.unibuc.ro and gigel.militaru@gmail.com}

\thanks{This work was supported by a grant of the Ministry of Research,
Innovation and Digitization, CNCS/CCCDI--UEFISCDI, project number
PN-III-P4-ID-PCE-2020-0458, within PNCDI III}


\begin{abstract}
Given a Jordan algebra $A$ and a vector space $V$, we describe and classify all Jordan algebras containing $A$ as a subalgebra of codimension
${\rm dim}_k (V)$ in terms of a non-abelian cohomological type object ${\mathcal J}_{A} \, (V, \, A)$.
 Any such algebra is isomorphic to a newly introduced object called \emph{unified
product} $A \, \natural \, V$. The crossed/twisted
product of two Jordan algebras are introduced as
special cases of the unified product and the role of the subsequent
problem corresponding to each such product is discussed. The
non-abelian cohomology ${\rm H}^2_{\rm nab} \, (V, \, A )$ associated to two
Jordan algebras $A$ and $V$ which classifies all extensions of $V$ by $A$ is also constructed.
Several applications and examples are given: we prove that ${\rm H}^2_{\rm nab} \, (k, \, k^n)$
is identified with the set of all matrices $D\in M_n(k)$ satisfying $2\, D^3 - 3 \, D^2 + D = 0$.
\end{abstract}

\maketitle

\section*{Introduction}

Jordan algebras have been first introduced in mathematical physics
by P. Jordan \cite{pjordan} in order to achieve an axiomatization
for the algebra of observables in quantum mechanics. Since then,
Jordan algebra theory has developed into a rich and fruitful one,
as evidenced by the various contexts where it makes an appearance
such as supersymmetry, the theory of superstrings, projective geometry, Lie algebras and algebraic groups,
representation theory or functional analysis (see \cite{AmWa, CDE, CoHo, Farnsworth, lopez, Mc} and the references
therein). One of the papers which generated a lot of interest and determined a turning point
for Jordan algebra theory is \cite{jordannumm} where all simple
finite-dimensional formally real algebras are classified. Another
important contribution is the paper of A. Albert \cite{albert}
where an example of an \emph{exceptional} Jordan algebra is
constructed: the $27$-dimensional algebra of self-adjoint $3\times
3$ matrices over the octonions. The study of Jordan algebras
experienced a major change in the second half of the last century
due mainly to the algebra school in Novosibirsk. In this context, major contributions which lead to a better understanding of
the field have been achieved by E. Zelmanov which extended results from the classical
theory of finite dimensional Jordan algebras to the infinite
dimensional case. In particular, all simple Jordan algebras,
including infinite-dimensional ones, have been classified by what
is now called \emph{Zelmanov's exceptional theorem} (\cite[Chapter
8]{Mc}): any simple exceptional Jordan algebra is an Albert
algebra of dimension $27$ over its center. The close relation
between Jordan algebras and Lie algebras is also highlighted by
the famous Kantor-Koecher-Tits construction \cite{lopez}.

The present paper is devoted to the study of the \emph{extending
structures problem} (ES-problem) introduced in \cite{am-2011} for
arbitrary categories and the \emph{extension problem} listed
below. In the context of Jordan algebras it comes down to the
following problem:

\textbf{Extending structures problem.} \textit{Let $A$ be a Jordan
algebra and $E$ a vector space containing $A$ as a subspace.
Describe and classify up to an isomorphism that stabilizes $A$
(i.e. acts as the identity on $A$) the set of all Jordan algebra
structures that can be defined on $E$ such that $A$ becomes a
Jordan subalgebra of $E$.}

Equivalently, if we fix $V$ a complement of $A$ in the vector space $E$, the ES-problem can be restated by asking for the description and classification of all
Jordan algebras which contain and stabilize $A$ as a subalgebra of
codimension equal to ${\rm dim} (V)$. As explained in
\cite{am-2014}, the ES problem generalizes and unifies two famous
and intensively studied open problems: the
\emph{extension problem} (introduced at the level of groups by
H\"{o}lder) and the \emph{factorization problem} (stemmed in group theory as well, in the work of Ore). The extension problem, formulated for
Jordan algebras, can be stated as follows:

\textbf{Extension problem.} \textit{Let $A$ and $V$ be two given
Jordan algebras. Describe and classify all extensions of $V$ by
$A$, i.e. all Jordan algebra $E$ that fit into an exact sequence
of Jordan algebras:} $\xymatrix{ 0 \ar[r] & A \ar[r]^{i} & E
\ar[r]^{p} & V \ar[r] & 0 }$.

The precise meaning of classification in the context of the extension problem appears in \seref{prel}.
The first result in this direction was proved by Jacobson \cite[Theorem 12]{jacb} who gave a partial
answer to the extension problem in the case of null extensions (i.e. such that
$a\cdot b := 0$, for all $a$, $b\in A$): in this case, similar to the group (or Lie algebra) case,
the extensions of $V$ by $A$ are parameterised by the second cohomology group ${\rm H}^2 \, (V,
\, A)$.

The paper is organized as follows: in \seref{prel} we recall some
basic concepts in the context of Jordan algebras.
\seref{unifiedprod} is devoted to the study of the extending structures problem
following the strategy we previously developed in \cite{am-2014,
am-2019}. \thref{main1}, the main result of this section, provides the theoretical answer to the ES-problem by constructing a \emph{non-abelian cohomological type object} ${\mathcal J}_{A} \,
(V, \, A )$ which classifies all Jordan algebras containing and stabilizing $A$ as a
subalgebra of codimension equal to ${\rm dim} (V)$.  Any such
algebra is isomorphic to a \emph{unified product} $A \, \natural
\, V$, a product associated to a Jordan algebra $A$ and a vector
space $V$ connected through two \emph{actions} and a generalized
\emph{cocycle}. It is worth mentioning that the unified
product generalizes spin factor Jordan algebras
(\exref{spinfactorm}) as first constructed in \cite{jordannumm}.
\thref{main1} offers the theoretical answer to the ES-problem.
However, computing the classifying object ${\mathcal J}_{A} \, (V,
\, A)$, for a given Jordan algebra $A$ and a vector space $V$, is
a highly non-trivial task. In \thref{clasdim1} we explicitly
compute ${\mathcal J}_{A} \, (k, \, A)$ which is the key step in
classifying finite dimensional \emph{flag Jordan algebras} introduced in \deref{flagex} as a generalization of supersolvable
Jordan algebras. Another application is given in \thref{recsiGal} which proves an Artin type theorem for Jordan algebras.
More precisely, we show that a Jordan algebra on which a finite group acts can be written as a twisted product between its subalgebra of invariants and a given complement. This is
a first step for possible future developments in the \emph{invariant
theory} of Jordan algebras.

One of the main special cases of the unified product is the crossed products of Jordan algebras, will be treated separately in \seref{Hilbertext}.
As in the group (or Lie algebra) case, the crossed product of two Jordan algebras plays a key role in the study of the extension problem. The theoretical
answer to the extension problem in its full generality is given in \thref{scchjor}:
all extensions of a given Jordan algebra $V$ by a Jordan algebra
$A$ are classified by the \emph{non-abelian cohomology} ${\rm
H}^2_{\rm nab} \, (V, \, A )$ which is explicitly constructed. This result is the non-abelian generalization of \cite[Theorem
12]{jacb}. The classification of finite dimensional Jordan
algebras was only recently initiated: in
\cite{km0} all three-dimensional Jordan algebras over $\RR$ are
classified while in \cite{martin} all four-dimensional Jordan
algebras over an algebraically closed field of characteristic
different than $2$ are classified. The classification of nilpotent Jordan algebras was completed up to dimension $5$ (\cite{HA, km}) and partial results are available in dimension $6$ \cite{AbM}. In order to obtain classification results in higher dimensions or for other classes of Jordan algebras, certain constructions such as the crossed or the bicrossed product might be useful. In fact, \coref{iterare} proves the crucial role played
by crossed products for the classification problem: any finite dimensional
Jordan algebra is isomorphic to an iterated crossed product
of the form $\Bigl( \cdots \bigl( (S_1 \# \, S_2)\# \, S_3 \bigl)
\# \, \cdots \# \, S_t\Bigl)$, where $S_i$ is a finite dimensional
simple Jordan algebra for all $i= 1,\cdots, t$ (recall that simple Jordan algebras have already been classified).
Computing explicitly the non-abelian cohomology ${\rm H}^2_{\rm
nab} \, (V, \, A )$ of two given Jordan algebras $A$ and $V$ is a really challenging
problem. In \exref{nonexcoho} we compute ${\rm H}^2_{\rm nab} \, (k, \, A )$ and in particular
we show that
$$
{\rm H}^2_{\rm nab} \, (k, \, k^n ) \cong \{D \in M_n (k) \,\,  | \,\, 2\, D^3 - 3 \, D^2 + D = 0  \}
$$
where $M_n (k)$ denotes the usual space of $n\times n$-matrices over $k$.

\section{Preliminaries}\selabel{prel}
Throughout, we work over a field $k$ of characteristic $\neq 2$. A bilinear map $f : W\times W \to V$ will be called
symmetric if $f (x, \, y) = f (y, \, x)$, for all $x$, $y \in W$. If $V$ is a vector space, $V^*$ denotes its dual. ${\rm End}_k \, (V)$ stands for the associative and
unital endomorphisms algebra of $V$. The associator of an arbitrary algebra $(A, \cdot)$ is
defined by $(x, \, y, \, z) := (x\cdot y)\cdot z - x \cdot (y\cdot z)$,
for all $x$, $y$, $z\in A$.

A \emph{Jordan algebra} is a vector space $A$ together with a
bilinear map $\cdot : A \times A \to A$, called multiplication,
such that for any $a$, $b\in A$ we have:
\begin{equation*}
a\cdot b = b\cdot a, \qquad (a^2 \cdot b) \cdot a = a^2 \cdot (b
\cdot a)
\end{equation*}
i.e. $\cdot$ is commutative and satisfies the Jordan identity $(a^2, \, b, \, a) = 0$, for all $a$, $b\in A$. For all unexplained notions pertaining to
Jordan algebra theory we refer the reader to \cite{jacb, Ko, Mc, Zh2}. A Jordan algebra $A$ is called \emph{abelian} if it has trivial
multiplication, i.e. $a \cdot b := 0$ for all $a$, $b\in A$. If $A$ is an associative algebra, then $A$ with the new product defined by $x \cdot y :=
2^{-1} (xy + yx)$, for all $x$, $y\in A$ becomes a Jordan algebra
and is denoted by $A^{+}$. In particular, ${\rm End}_k \, (V)
^{+}$ is a Jordan algebra for any vector space $V$. A Jordan
algebra is said to be special if it is isomorphic to a subalgebra
of $A^{+}$, for an associative algebra $A$. Non-special Jordan
algebras are called exceptional. We denote by ${\rm Aut}_{\rm
J}(A)$ the automorphism group of the Jordan algebra $A$. If $A$ is a Jordan algebra, then it is well
known \cite{kpa} that the following \emph{polarization relation} holds for
any $x$, $y$, $z\in A$:
\begin{equation}\eqlabel{polarizare}
(x^2, \, y, \, z) + 2 \, (x\cdot z, \, y, x) = 0
\end{equation}
This can be easily proved by taking $a : = x + \lambda z$, $b:= y$, for any $\lambda \in k^*$, in the Jordan identity.
For an arbitrary non-associative algebra $(A, \cdot)$ we define the so-called \emph{polarization map} of $A$ as follows:
\begin{equation}\eqlabel{polarmap}
P: A\times A \times A \to A, \qquad P(x, y, z) := (x^2, \, y, \, z) + 2 \, (x\cdot z, \, y, x)
\end{equation}
for all $x$, $y$, $z\in A$.

\begin{remark} \relabel{impbase}
Unlike the case of associative or Lie algebras, showing that the Jordan algebra identity holds is rather cumbersome. More precisely, the major drawback in the case of Jordan algebras is that one needs to check the Jordan identity on the entire underlying vector space not only on the elements of its $k$-basis.
\end{remark}

The unusual situation in \reref{impbase} is explained by the following result that will be used later on in the proof of \thref{1}:

\begin{proposition} \prlabel{missingrel}
Let $(W, \circ)$ be a commutative algebra such that $W = A \, + \, V$, for two subspaces $A$, $V \subseteq W$. Assume that the Jordan identity
$(\alpha^2, \, \beta, \, \alpha) = 0$ holds for all $\alpha$, $\beta \in A \cup V$. Then, $(W, \circ)$ is a Jordan algebra if and only if for any $a$, $b \in A$, $x$, $y\in V$ we have:
\begin{equation}\eqlabel{missing}
P(a, \, b, \, x) + P(x, \, b, \, a) + P(a, \, y, \, x) + P(x, \, y, \, a) = 0
\end{equation}
where $P: W^3 \to W$ is the polarization map of $W$.
\end{proposition}

\begin{proof}
For arbitrary elements $X$, $Y\in W$ we write $X = a + x$ and $Y = b + y$, for some $a$, $b \in A$ and $x$, $y\in V$. Then, it can be easily proved that:
\begin{eqnarray*}
(X^2 \circ Y) \circ X &=& (a^2 \circ b) \circ a + (a^2 \circ b) \circ x + (a^2 \circ y) \circ a + (a^2 \circ y) \circ x + \\
&+& (x^2 \circ b) \circ a + (x^2 \circ b) \circ x + (x^2 \circ y) \circ a + (x^2 \circ y) \circ x + \\
&+& 2 \, \bigl((a\circ x) \circ b\bigl) \circ a + 2 \, \bigl((a\circ x) \circ b\bigl) \circ x + 2 \, \bigl((a\circ x) \circ y\bigl) \circ a + \\
&+& 2 \, \bigl((a\circ x) \circ y\bigl) \circ x
\end{eqnarray*}
and
\begin{eqnarray*}
X^2 \circ (Y \circ X) &=& a^2 \circ (b \circ a) + a^2 \circ (b \circ x) + a^2 \circ (y \circ a) + a^2 \circ (y \circ x) + \\
&+& x^2 \circ (b \circ a) + x^2 \circ (b \circ x) + x^2 \circ (y \circ a) + x^2 \circ (y \circ x) + \\
&+& 2 \, (a\circ x) \circ (b \circ a) + 2 \, (a\circ x) \circ (b \circ x) + 2 \, (a\circ x) \circ (y \circ a) + \\
&+& 2 \, (a\circ x) \circ (y \circ x)
\end{eqnarray*}
Taking into account that the Jordan identity
$(\alpha^2, \, \beta, \, \alpha) = 0$ holds for any pair of elements $\alpha$, $\beta \in A \cup V$, we obtain that
$(X^2 \circ Y) \circ X = X^2 \circ (Y \circ X)$ if and only if
\begin{eqnarray*}
&& (a^2, \, b, \, x) + (a^2, \, y, \, x) + (x^2, \, b, \, a) + (x^2, \, y, \, a) + \\
&& + 2 \, (a \circ x, \, b, \, a) + 2 \, (a \circ x, \, b, \, x) + 2 \, (a \circ x, \, y, \, a) + 2 \, (a \circ x, \, y, \, x) = 0
\end{eqnarray*}
which is equivalent to \equref{missing}, using the formula \equref{polarmap}. The proof is now finished.
\end{proof}

\begin{definition}\delabel{moduleJ}
A \emph{right action} of a Jordan algebra $A$ on a vector space
$V$ is a bilinear map $\triangleleft : V \times A \to V$ such that
for any $a\in A$ and $x\in V$ we have:
\begin{eqnarray} \eqlabel{actiunedr}
(x \triangleleft a^2) \triangleleft a = (x \triangleleft a)
\triangleleft a^2
\end{eqnarray}
Similarly, a \emph{left action} of $A$ on $V$ is a bilinear map
$\triangleright : A \times V \to V$ such that for any $a\in A$ and
$x\in V$ we have:
\begin{eqnarray} \eqlabel{actiunest}
a \triangleright  (a^2 \triangleright x) = a^2 \triangleright (a
\triangleright x)
\end{eqnarray}
\end{definition}

The canonical maps $\triangleright : A \times A \to A$ and
$\triangleright : A \times A^* \to A^*$ given for any $a$, $b\in
A$ and $a^* \in A^*$ by:
\begin{equation} \eqlabel{calduramare}
a \triangleright b := a \cdot b, \qquad (a \triangleright a^*) (b)
:= a^* (a\cdot b)
\end{equation}
are left actions of $A$ on $A$ and $A^* = {\rm Hom}_k
\, (A, \, k)$, respectively.

\begin{remarks} \relabel{Eilenberg}
1. Since a Jordan algebra $A$ is commutative any right/left action
is a left/right action too; nevertheless, we shall use both types
of actions throughout the paper in order to indicate precisely the
position of the acting algebra. The set of all right/left actions
of $A$ on $V$ is in bijection with the set of all linear maps
$\varphi : A \to {\rm End}_k \, (V)$ such that $\varphi (a) \circ
\varphi (a^2) = \varphi (a^2) \circ \varphi (a)$, for all $a\in A$
in the endomorphism algebra of $V$. In certain references, such a
map $\varphi : A \to {\rm End}_k \, (V)$ is called a \emph{Jacobson representation} of $A$ on $V$.
Moreover, the left/right action condition
\equref{actiunest} appears as one of the compatibilities in the definition of a
Jordan bimodule or a representation of a Jordan algebra \cite{CDDV, KS}.

2. Let $\varphi: A \to {\rm End}_k (V)^{+}$ be a morphism of Jordan
algebras. If we denote $\varphi (a) (x) = a \triangleright x$, we
have that $(a\cdot b) \triangleright x = 2^{-1} \bigl(a
\triangleright (b \triangleright x) + b \triangleright (a
\triangleright x)\bigl)$, for all $a$, $b\in A$ and $x\in V$. In
particular, for $b = a$ we obtain that $a^2 \triangleright x = a
\triangleright (a \triangleright x)$. Using this relation we can
easily see that \equref{actiunest} holds, i.e. any Jordan algebra
map $\varphi: A \to {\rm End}_k (V)^{+}$ induces a left action of
$A$ on $V$, as expected.
\end{remarks}

\begin{definition} \delabel{echivextedna}
Let $A$ and $V$ be two Jordan algebras. An \emph{extension} of $V$
by $A$ is a triple $(E, \, i, \, p)$ consisting of a Jordan
algebra $E$ and two morphisms of Jordan algebras that fit into an
exact sequence of Jordan algebras
\begin{eqnarray*}
\xymatrix{ 0 \ar[r] & A \ar[r]^{i} & E  \ar[r]^{p} & V \ar[r] & 0
}
\end{eqnarray*}
Two extensions $(E, \, i, \, p)$ and $(E', \, i', \, p')$ of $V$
by $A$ are called \emph{cohomologous} and we denote this by $(E, \,
i, \, p) \approx (E', \, i', \, p')$ if there exists a Jordan
algebra map $\varphi : E \to E'$ such that the following diagram is commutative:
\begin{eqnarray} \eqlabel{diagramaext}
\xymatrix {& A \ar[r]^{i} \ar[d]_{Id} & {E}
\ar[r]^{p} \ar[d]^{\varphi} & V \ar[d]^{Id}\\
& A \ar[r]^{i'} & {E'}\ar[r]^{p' } & V}
\end{eqnarray}
We denote by ${\rm Ext} \, (V, \,  A)$ the set of
all equivalence classes of extensions of $V$ by $A$ via $\approx$.
We say that $\varphi: E \to E'$ \emph{stabilizes} $A$ (resp.
\emph{co-stabilizes} $V$) if the left square (resp. the right
square) of diagram \equref{diagramaext} is commutative.
\end{definition}

Any Jordan algebra map $\varphi : E \to E'$ which makes diagram
\equref{diagramaext} commutative is an isomorphism and thus
$\approx$ is indeed an equivalence relation on the class of all
extensions of $V$ by $A$. Hilbert's extension problem,
formulated for Jordan algebras, is the following difficult open question: \emph{For two given Jordan algebras $A$ and $V$ compute explicitly
the classifying object ${\rm Ext} \, (V, \, A)$.}

\section{Extending structures problem}\selabel{unifiedprod}

This section aims at providing a theoretical answer to the
extending structures problem for Jordan algebras. To this end, we introduce the following:

\begin{definition} \delabel{echivextedn}
Let $A$ be a Jordan algebra and $E$ a vector space containing $A$
as a subspace. Two Jordan algebra structures on $E$ denoted by $\cdot_E $
and $\cdot_E^{'}$, both containing $A$ as a subalgebra, are called
\emph{equivalent}, and we denote this by $(E, \, \cdot_E) \equiv
(E, \, \cdot_E^{'})$, if there exists a Jordan algebra isomorphism
$\varphi: (E, \, \cdot_E) \to (E, \, \cdot_E^{'})$ which
stabilizes $A$, that is $\varphi (a) = a$, for all $a\in A$. We
denote by ${\rm Jext} \, (E,\, A)$ the set of all equivalence
classes on the set of all Jordan algebras structures on $E$
containing $A$ as a subalgebra via the equivalence relation
$\equiv$.
\end{definition}

The set ${\rm Jext} \, (E, \, A)$ defined above is the object responsible for the classification part of the extending structures problem. We will show that ${\rm Jext} \, (E, \, A)$ is parameterized by a \emph{cohomological type object} which we will be able to construct explicitly.

\begin{definition}\delabel{exdatum}
Let $A = (A, \cdot)$ be a Jordan algebra and $V$ a vector space.
An \textit{extending datum of $A$ through $V$} is a system
$\Omega(A, \,  V) = \bigl(\triangleleft, \, \triangleright, \, f,
\,  \cdot_V \bigl)$ consisting of four bilinear maps
$$
\triangleleft : V \times A \to V, \quad \triangleright : V \times
A \to A, \quad f: V\times V \to A, \quad \cdot_V : V\times V \to V
$$
Let $\Omega(A, \, V) = \bigl(\triangleleft, \, \triangleright, \,
f, \, \cdot_V \bigl)$ be an extending datum. We denote by $ A \,
\,\natural \,_{\Omega(A, \, V)} V = A \, \,\natural \, V$ the
vector space $A \, \times V$ together with the bilinear map $
\circ : (A \times V) \times (A \times V) \to A \times V$ defined
by:
\begin{equation}\eqlabel{brackunif}
(a, x) \circ (b, y) := \bigl( a \cdot b + x \triangleright b +
y\triangleright a + f(x, \, y), \,\,  x\triangleleft b +
y\triangleleft a + x\cdot_V \, y \bigl)
\end{equation}
for all $a$, $b \in A$ and $x$, $y \in V$. The object $A
\,\natural \, V$ is called the \textit{unified product} of $A$ and
$V$ if it is a Jordan algebra with the multiplication given by
\equref{brackunif}. In this case the extending datum $\Omega(A, \,
V) = \bigl(\triangleleft, \, \triangleright, \, f, \, \cdot_V
\bigl)$ is called a \textit{Jordan extending structure} of $A$
through $V$. The maps $\triangleleft$ and $\triangleright$ are
called the \textit{actions} of $\Omega(A, \, V)$ and $f$ is called
the (non-abelian) \textit{cocycle} of $\Omega(A, \, V)$.
\end{definition}

In the sequel we use the following convention: if any of the maps of an extending datum $\Omega(A, \, V) =
\bigl(\triangleleft, \, \triangleright, \, f, \,  \cdot_V \bigl)$
is trivial then we will write the quadruple
$\bigl(\triangleleft, \, \triangleright, \, f, \, \cdot_V \bigl)$ without it.  
The first natural example of a unified product was first constructed in \cite{jordannumm} related to the classification of simple
finite-dimensional formally real Jordan algebras.

\begin{example} \exlabel{spinfactorm}
Let $V$ be a vector space equipped with a symmetric bilinear form
$f: V\times V \to k$ and let $J(V, f) : = k \times V$ with the
multiplication
\begin{equation*}
(a, x) \circ (b, y) := \bigl(a b + f(x, \, y), \,\, bx
 + a y \bigl)
\end{equation*}
for all $a$, $b \in k$ and $x$, $y \in V$. Then $J(V, f)$ is a
Jordan algebra and $J(V, f) = k \, \,\natural \, V$ is a unified
product of the Jordan algebras $k$ (with the usual multiplication) and
$V$: the maps $\triangleright$ and $\cdot_V$ are both trivial and
the (right) action $\triangleleft$ of $k$ on $V$ is just the $k$-vector space
structure on $V$. The Jordan algebra $J(V,
f)$ is called \emph{spin factor} Jordan algebra, the
\emph{Clifford Jordan algebra} or the \emph{Jordan algebra
associated to a symmetric bilinear form}. Moreover, $J(V, f)$ is
special and if $f$ is nondegenerate and ${\rm dim}_k (V) > 1$ then
$J(V, f)$ is simple \cite{Zh2}.
\end{example}

Let $\Omega(A, \, V)$ be an extending datum of $A$ through $V$. It will be worth the while to write down the following relations which hold in  $A \,\natural \, V$:

\begin{eqnarray}
(a, 0) \circ (b, y) &=& \bigl(a \cdot b + y \triangleright
a, \,\,  y \triangleleft a \bigl) \eqlabel{001}\\
(0, x) \circ (b, y) &=& \bigl( x \triangleright b + f(x, y), \,\,
x \triangleleft b + x \cdot_V y \bigl) \eqlabel{002}
\end{eqnarray}
for all $a$, $b \in A$ and $x$, $y \in V$.

\begin{theorem}\thlabel{1}
Let $\Omega(A, \,  V) = \bigl(\triangleleft, \, \triangleright, \,
f, \, \cdot_V \bigl)$ be an extending datum of a Jordan algebra $A
= (A, \cdot)$ through a vector space $V$. The following statements
are equivalent:

$(1)$ $A \,\natural \, V$ is a unified product;

$(2)$ The following compatibilities hold for any $a$, $b \in A$,
$x$, $y \in V$:
\begin{enumerate}
\item[(E1)] $f: V\times V \to A$ and $\cdot_V : V\times V
\to V$ are symmetric maps; \\
\item[(E2)] $\triangleleft : V \times A \to V$ is a right action
of $A$ on $V$, i.e. $(x \triangleleft a^2) \triangleleft a = (x
\triangleleft a) \triangleleft a^2$;\\
\item[(E3)] $a \cdot (x \triangleright a^2) + (x \triangleleft a^2
)\triangleright a$ = $a^2 \cdot (x \triangleright a) + (x
\triangleleft a )\triangleright a^2$;\\
\item[(E4)] $x \triangleright \bigl( f(x, x) \cdot a \bigl) + \, x
\triangleright (x^2 \triangleright a) + f(x^2 \triangleleft a, \,
x) \, \, = $

$f(x, x) \cdot (x \triangleright a) + x^2 \triangleright (x
\triangleright a) + (x \triangleleft a) \triangleright f(x, x) +
f(x^2, \, x \triangleleft a)$;\\
\item[(E5)] $x \triangleleft \bigl( f(x, x) \cdot a \bigl) + x
\triangleleft (x^2 \triangleright a) + (x^2 \triangleleft a) \,
\cdot_V \, x \,\, =$

$x^2 \triangleleft (x \triangleright a) + (x \triangleleft a)
\triangleleft f(x, x) + x^2 \, \cdot_V \, (x \triangleleft a)$;\\
\item[(E6)] $x \triangleright \bigl(y \triangleright f(x, x)
\bigl) + \, x \triangleright f(x^2, y) + f\bigl( y \triangleleft
f(x, x) + x^2 \,\cdot_V \, y, \, x) \bigl) \, \, = $

$f(x, x) \cdot f(x, y) + x^2 \triangleright f(x, y) + (x\cdot_ V
\, y) \triangleright f(x, x) + f(x^2, \, x\cdot_V \, y)$;\\
\item[(E7)] $x \triangleleft \bigl(y \triangleright f(x, x) \bigl)
+ \, x \triangleleft f(x^2, \, y) + \bigl(y \triangleleft f(x, x)
\bigl) \, \cdot_V \, x + (x^2 \cdot_V \, y) \cdot_V \, x \, \, = $

$x^2 \triangleleft f(x, y) + (x\cdot_V \, y)\triangleleft f(x, x)
+ x^2 \cdot_V \, (x \cdot_V \, y)$; \\
\item[(E8)] The following \emph{missing relations} hold in the algebra $(A \,\natural \, V, \, \circ)$:
\begin{equation}\eqlabel{missingrel}
P(a, \, b, \, x) + P(x, \, b, \, a) + P(a, \, y, \, x) + P(x, \, y, \, a) = 0
\end{equation}
where $P: (A \,\natural \, V)^3 \to A \,\natural \, V$ is the polarization map of $A \,\natural \, V$ and
we identify $\alpha = (\alpha, 0)$, $z = (0, z)$, for any $\alpha \in A$ and $z \in V$.
\end{enumerate}
\end{theorem}

\begin{proof}
Since the proof is based on a rather cumbersome but straightforward computation we will only indicate its main steps. First, we can easily prove that the multiplication defined by
\equref{brackunif} is commutative if and only if $f: V\times V \to
A$ and $\cdot_V : V\times V \to V$ are both symmetric maps, i.e.
(E1) holds. From now on we will assume that (E1) holds. Thus $A
\,\natural \, V$ is a Jordan algebra if and only if Jordan's
identity holds, i.e.:
\begin{equation}\eqlabel{005}
\{\bigl( (a, x) \circ (a, x) \bigl) \, \circ \, (b, y) \} \circ
(a, x) = \bigl( (a, x) \circ (a, x) \bigl) \, \circ \, \{(b, y)
\circ (a, x) \}
\end{equation}
for all $a$, $b \in A$ and $x$, $y \in V$. First, using \equref{001} we can easily notice that
\equref{005} holds for the pair $(a, 0)$, $(b, 0)$. Next, we can
prove that \equref{005} holds for $(a, 0)$, $(0, x)$ if and only
if (E2) and (E3) hold. Secondly, we can prove that \equref{005}
holds for $(0, x)$, $(a, 0)$ if and only if (E4) and (E5) hold.
Finally, \equref{005} holds for $(0, x)$, $(0, y)$ if and only if
(E6) and (E7) hold.

To conclude, the compatibility conditions
(E1)-(E7) are equivalent to the fact that the multiplication on $A \,\natural \, V$ is commutative and
the Jordan identity \equref{005} holds for all elements $\alpha$, $\beta \in (A\times \{0\}) \cup (\{0\} \times V)$.
Since $A \,\natural \, V = (A\times \{0\}) + (\{0\} \times V)$ we obtain, using \prref{missingrel}, that
$A \,\natural \, V$ is a Jordan algebra if and only if (E8) holds and the proof is finished.
\end{proof}

\begin{remark} \relabel{explicitlipsa}
The compatibilities in \equref{missingrel} are imposed by \reref{impbase} and axioms of this kind are not present in the case of Lie nor associative algebras (see \cite{am-2014, am-2016}). For this reason, they are called 'missing relations'.
\end{remark}

The second example of a unified product is the bicrossed product of
two Jordan algebras:

\begin{example} \exlabel{bicrossed}
Let $\Omega(A, \,  V) = \bigl(\triangleleft, \, \triangleright, \,
f, \, \cdot_V \bigl)$ be an extending datum of a Jordan algebra
$A$ through a vector space $V$ such that $f := 0$ is the trivial map. Then, using
\thref{1}, we obtain that $\Omega(A, \,  V) = \bigl(\triangleleft, \, \triangleright, \, f := 0, \, \cdot_V \bigl)$ is a
Jordan extending structure of $A$ through $V$ if and only if $(V, \,
\cdot_V)$ is a Jordan algebra (by (E1) and (E7)) and $(A, \, V, \, \triangleleft, \, \triangleright)$ is a \emph{matched pair} of Jordan algebras in the sense of \cite[Definition 2.1]{am-2022}.
\end{example}

Let $A$ be a Jordan algebra $A$ and $V$ a vector space. In the sequel, we shall denote by ${\mathcal J} {\mathcal E} {\mathcal S} \, (A, \, V)$ the set of
all Jordan extending structures of $A$ through $V$ (or, equivalently, all
systems $\Omega(A, \, V) = \bigl(\triangleleft, \, \triangleright,
\, f, \, \cdot_V \bigl)$ satisfying the compatibility conditions
(E1)-(E8) of \thref{1}). Note that the set  ${\mathcal J}{\mathcal E}
{\mathcal S} \, (A, \, V)$ contains at least one element, namely the extending
structure $\Omega (A, \, V) = \bigl(\triangleleft, \,
\triangleright, \, f, \, \cdot_V \bigl)$ consisting only of trivial maps and whose associated unified product is the direct product between $A$ and
the abelian Jordan algebra $V$.

Let $\Omega(A, \, V) = \bigl(\triangleleft, \, \triangleright, \,
f,\, \cdot_V \bigl) \, \in {\mathcal J} {\mathcal E} {\mathcal S}
\,  (A, \, V)$ be a Jordan extending structure and $A \,\natural
\, V$ the associated unified product. We have an injective Jordan algebra map defined as follows:
$$
i_{A}: A \to A \,\natural \, V, \qquad i_{A}(a) = (a, \, 0)
$$
which allows us to see $A$ as a
Jordan subalgebra of $A \,\natural \, V$ via the
identification $A \cong i_{A}(A) = A \times \{0\}$. Furthermore, the converse also holds: any Jordan algebra structure on a vector space $E$
containing $A$ as a Jordan subalgebra is isomorphic to a unified
product.

\begin{theorem}\thlabel{classif}
Let $A = (A, \cdot)$ be a Jordan algebra, $E$ a vector space and
$\cdot_E$ a Jordan algebra structure on $E$ containing $A$ as a
Jordan subalgebra. Then there exists a Jordan extending structure
$\Omega(A, \, V) = \bigl(\triangleleft, \, \triangleright, \, f,
\, \cdot_V \bigl)$ of $A$ through a subspace $V$ of $E$ and an
isomorphism of Jordan algebras $(E, \cdot_E) \cong A \,\natural \,
V$ that stabilizes $A$.
\end{theorem}

\begin{proof} Let $\cdot_E$ be Jordan algebra structure on $E$
containing $A$ as a Jordan subalgebra, i.e. $a\cdot_E b = a \cdot
b$, for all $a$, $b\in A$. Working over a field allows us to find a linear map $p: E \to A$ such that $p(a) = a$, for all $a
\in A$. Consequently, $V := {\rm Ker}(p)$ is a
complement of $A$ in $E$, i.e. $E = A \oplus V$. Using the
retraction $p$ we can now define the extending datum of $A$
through $V$ as follows:
\begin{eqnarray}
\triangleright = \triangleright_p : V \times A \to
A, \qquad x \triangleright a &:=& p (x \cdot_E a ) = p(a \cdot_E x ) \eqlabel{bala1}\\
\triangleleft = \triangleleft_p: V \times A \to V,
\qquad x \triangleleft a &:=& x \cdot_E a - p (x \cdot_E a) \eqlabel{bala2} \\
f = f_p: V \times V \to A, \qquad f(x, y) &:=&
p (x \cdot_E y)  \eqlabel{bala3} \\
\cdot_V = (\cdot_V)_p: V \times V \to V, \qquad x\cdot_V \, y &:=&
x\cdot_E \, y  - p (x \cdot_E \, y) \eqlabel{bala4}
\end{eqnarray}
for any $a \in A$ and $x$, $y\in V$. First of all, it is
straightforward to see that the above maps are well defined
bilinear maps: $x \triangleleft a \in V$ and $x\cdot_V \, y \in V
= {\rm Ker}(p)$, for all $x$, $y \in V$ and $a \in A$. We will
show that $\Omega(A, \, V) = \bigl(\triangleleft, \,
\triangleright, \, f, \, \cdot_V \bigl)$ is a Jordan extending
structure of $A$ through $V$ and $ \varphi: A \, \natural \, V \to
E$, $\varphi(a, x) := a + x$ is an isomorphism of Jordan algebras
that stabilizes $A$. The strategy we use, relaying on \thref{1},
is the following: $\varphi: A \times V \to E$, $\varphi(a, \, x)
:= a + x$ is a linear isomorphism between the Jordan algebra $E =
(E, \cdot_E)$ and the direct product of vector spaces $A \times V$
with the inverse given by $\varphi^{-1}(y) := \bigl(p(y), \, y -
p(y)\bigl)$, for all $y \in E$. Thus, there exists a unique Jordan
algebra structure on $A \times V$ such that $\varphi$ is an
isomorphism of Jordan algebras and this unique multiplication on
$A \times V$ is given for any $a$, $b \in A$ and $x$, $y\in V$ by:
$$
(a, x) \circ (b, y) := \varphi^{-1} \bigl( \varphi(a, x) \,
\cdot_E \,  \varphi(b, y) \bigl)
$$
We are now left to prove that the above multiplication coincides
with the one associated to the system $\bigl(\triangleleft_p, \,
\triangleright_p, \, f_p, \, (\cdot_V)_p \bigl)$ as defined by
\equref{brackunif}. Indeed, using intensively the commutativity of
$\cdot_E$, for any $a$, $b \in A$ and $x$, $y\in V$ we have:
\begin{eqnarray*}
(a, x) \circ (b, y) &=& \varphi^{-1} \bigl(\varphi(a, x) \,
\cdot_E \,  \varphi(b, y)\bigl) = \varphi^{-1} \bigl( (a+x)
\cdot_E (b+y) \bigl) \\
&=& \varphi^{-1} (a\cdot_E b + a\cdot_E y + x\cdot_E b + x\cdot_E
y) \\
&=& \bigl( p(a\cdot b) + p(a \cdot_E \, y ) + p (x \cdot_E \, b) +
p(x \cdot_E \, y), \, \\
&& a\cdot b + a\cdot_E y + x\cdot_E b + x\cdot_E y - p(a\cdot b) -
p(a \cdot_E \, y ) - \\
&& - p (x \cdot_E \, b) - p(x \cdot_E \, y)  \bigl)\\
&=& \Bigl(a \cdot b + x \triangleright b + y \triangleright a +
f(x, \, y), \,\,  x \triangleleft b + y \triangleleft a + x\cdot_V
\, y \Bigl)
\end{eqnarray*}
as desired. Moreover, the following diagram
\begin{eqnarray*}
\xymatrix {& A \ar[r]^{i_{A}} \ar[d]_{Id} &
{A \,\natural \, V} \ar[d]^{\varphi} \\
& A \ar[r]^{i} & {E} }
\end{eqnarray*}
is obviously commutative which shows that $\varphi$ stabilizes $A$ and this finishes the
proof.
\end{proof}

\thref{classif} shows that the classification of all Jordan algebra
structures on $E$ that contain $A$ as a subalgebra comes down to the
classification of the unified products $A \,\natural \, V$
for a given complement $V$ of $A$ in $E$. In order to describe the classifying sets ${\rm Jext} \, (E, A)$
introduced in \deref{echivextedn}, we need the following:

\begin{lemma} \lelabel{morfismuni}
Let $\Omega(A, \, V) = \bigl(\triangleleft, \, \triangleright, \,
f, \, \cdot_V \bigl)$ and $\Omega'(A, \, V) = \bigl(\triangleleft
', \, \triangleright ', \, f', \, \cdot_V^{'} \bigl)$ be two
Jordan algebra extending structures of a Jordan algebra $A = (A,
\cdot)$ through $V$ and $ A \,\natural \, V$, respectively $ A
\,\natural \, ' V$, the associated unified products. Then there
exists a bijection between the set of all morphisms of Jordan
algebras $\psi: A \,\natural \, V \to A \,\natural \, ' V$ which
stabilize $A$ and the set of pairs $(r, v)$, where $r: V \to A$,
$v: V \to V$ are two linear maps satisfying the following
compatibility conditions for any $a \in A$, $x$, $y \in V$:
\begin{enumerate}
\item[(M1)] $v(x) \triangleleft ' a = v(x \triangleleft a)$;
\item[(M2)] $ v(x) \triangleright ' a = r(x \triangleleft a) + x
\triangleright a - a \cdot r(x)$;

\item[(M3)] $v(x \cdot_V \, y) = v(x)\cdot_V ^{'} \, v(y) + v(x)
\triangleleft ' r(y) + v(y) \triangleleft ' r(x)$;

\item[(M4)] $r(x \cdot_V \, y) = r(x) \cdot r(y) + v(x)
\triangleright ' r(y) + v(y) \triangleright ' r(x) + f'
\bigl(v(x), v(y)\bigl) - f(x, y)$
\end{enumerate}
Under the above bijection the morphism of Jordan algebras $\psi =
\psi_{(r, v)}: A \,\natural \, V \to A \,\natural \, ' V$
corresponding to $(r, v)$ is given for any $a \in A$ and $x \in V$
by:
$$
\psi(a, \, x) = (a + r(x), \, v(x))
$$
Moreover, $\psi = \psi_{(r, v)}$ is an isomorphism if and only if
$v: V \to V$ is bijective and $\psi_{(r, v)}$ co-stabilizes $V$ if
and only if $v = {\rm Id}_V$, the identity of $V$.
\end{lemma}

\begin{proof}
To start with, it can be easily seen that if a linear map $\psi: A \,\natural \, V \to A \,\natural \, ' V$ makes the following diagram commutative:
$$
\xymatrix {& {A} \ar[r]^{i_{A}} \ar[d]_{Id_{A}}
& {A \,\natural \, V}\ar[d]^{\psi}\\
& {A} \ar[r]^{i_{A}} & {A \,\natural \, ' V}}
$$
then there exists two linear maps $r: V \to A$, $v: V \to
V$ such that $\psi(a, x) = (a + r(x), v(x))$, for all $a \in A$,
and $x \in V$.

Let $\psi = \psi_{(r, v)}$ be such a linear map,
i.e. $\psi(a, x) = (a + r(x), v(x))$, for some linear maps $r: V
\to A$, $v: V \to V$. We will prove that $\psi$ is a morphism of
Jordan algebras if and only if the compatibility conditions
(M1)-(M4) hold. To this end, it is enough to prove that the
compatibility
\begin{equation}\eqlabel{Liemap}
\psi \bigl((a, x) \circ (b, y) \bigl) = \psi(a, x) \circ' \psi(b,
y)
\end{equation}
holds on all generators of $A \,\natural \, V$. We leave out the detailed computations and only indicate the main steps of this verification. First, it is easy to see that \equref{Liemap} holds for the pair $(a, 0)$, $(b, 0)$, for all $a$, $b \in A$. Secondly, we
can prove that \equref{Liemap} holds for the pair $(a, 0)$, $(0,
x)$ if and only if (M1) and (M2) hold. Finally, \equref{Liemap}
holds for the pair $(0, x)$, $(0, y)$ if and only if (M3) and (M4)
hold. The last statement is elementary: we just note that if $v: V
\to V$ is bijective, then $\psi_{(r, v)}$ is an isomorphism of
Jordan algebras with the inverse given for any $b \in A$ and $y
\in V$ by:
$$
\psi_{(r, v)}^{-1}(b, \, y) = \bigl(b - r(v^{-1}(y)), \,
v^{-1}(y)\bigl)
$$
This finishes the proof.
\end{proof}

Arising from \leref{morfismuni} the following concept will be used for the classification of unified products:

\begin{definition}\delabel{echiaa}
Let $A$ be a Jordan algebra and $V$ a vector space. Two Jordan
algebra extending structures of $A$ by $V$, $\Omega(A, \, V) =
\bigl(\triangleleft, \, \triangleright, \, f, \, \cdot_V \bigl)$
and $\Omega'(A, \, V) = \bigl(\triangleleft ', \, \triangleright
', \, f', \, \cdot_V^{'} \bigl)$ are called \emph{equivalent}, and
we denote this by $\Omega(A, \, V) \equiv \Omega'(A, \, V)$, if
there exists a pair of linear maps $(r, v)$, where $r: V \to A$
and $v \in {\rm Aut}_{k}(V)$ such that $\bigl(\triangleleft ', \,
\triangleright ', \, f', \, \cdot_V^{'} \bigl)$ is defined via
$\bigl(\triangleleft, \, \triangleright, \, f, \, \cdot_V \bigl)$
using $(r, v)$ as follows:
\begin{eqnarray*}
x \triangleleft ' a &=& v \bigl(v^{-1}(x) \triangleleft a\bigl)  \\
x \triangleright ' a &=& r \bigl(v^{-1}(x) \triangleleft a\bigl) +
\, v^{-1}(x) \triangleright a - a \cdot r\bigl(v^{-1}(x)\bigl)\\
f'(x, \, y) &=& f \bigl(v^{-1}(x), \, v^{-1}(y)\bigl) +
\, r \bigl(v^{-1}(x) \, \cdot_V \, v^{-1}(y)\bigl) + \, r\bigl(v^{-1}(x)\bigl) \, \cdot \, r\bigl(v^{-1}(y)\bigl)\\
&& - \, r\bigl(v^{-1}(x) \triangleleft r
\bigl(v^{-1}(y)\bigl)\bigl) - \, v^{-1}(x) \triangleright r
\bigl(v^{-1}(y)\bigl) - \, r\bigl(v^{-1}(y) \triangleleft r \bigl(v^{-1}(x)\bigl)\bigl)\\
&&  - \, v^{-1}(y) \triangleright r \bigl(v^{-1}(x)\bigl) \\
x \cdot_V^{'} y &=& v \bigl( v^{-1}(x) \, \cdot_V \, v^{-1}(y)
\bigl) - \, v \bigl(v^{-1}(x) \triangleleft r
\bigl(v^{-1}(y)\bigl)\bigl)
 - \, v \bigl(v^{-1}(y) \triangleleft r \bigl(v^{-1}(x)\bigl)\bigl)
\end{eqnarray*}
for all $a \in A$, $x$, $y \in V$.
\end{definition}

By putting together all the results proved in this section we obtain the following theoretical answer to the extending structures
problem for Jordan algebras:

\begin{theorem}\thlabel{main1}
Let $A$ be a Jordan algebra, $E$ a vector space that contains $A$
as a subspace and $V$ a complement of $A$ in $E$. Then:

$(1)$ $\equiv$ from \deref{echiaa} is an equivalence
relation on the set $ {\mathcal J} {\mathcal E} {\mathcal S} (A,
\, V)$ of all Jordan extending structures of $A$ through $V$. We
denote by ${\mathcal J}_{A} \, (V, \, A ) := {\mathcal J}
{\mathcal E} {\mathcal S}  (A, \, V)/ \equiv $, the quotient set.

$(2)$ The map
$$
{\mathcal J}_{A} \, (V, \, A) \to {\rm Jext} \, (E, \, A), \qquad
\overline{(\triangleleft, \triangleright, f, \, \cdot_V )}
\rightarrow \bigl(A \,\natural \, V, \, \circ \bigl)
$$
is bijective, where $\overline{(\triangleleft, \triangleright, f,
\cdot_V)}$ is the equivalence class of $(\triangleleft,
\triangleright, f, \cdot_V)$ via $\equiv$.
\end{theorem}

\begin{proof} We observe that $\Omega(A, \,  V) \equiv \Omega'(A, \, V)$
in the sense of \deref{echiaa} if and only if there exists an
isomorphism of Jordan algebras $\psi: A \,\natural \, V \to A
\,\natural \, ' V$ which stabilizes $A$. Therefore, $\equiv$ is an
equivalence relation on the set ${\mathcal J} {\mathcal
E}{\mathcal S} (A, \, V)$ of all Jordan algebra extending
structures $\Omega(A, \, V)$. Now, the conclusion follows from
\thref{1}, \thref{classif} and \leref{morfismuni}.
\end{proof}

The main application of the results proven in this section will be
given in the next section. Here we consider only two of them:
the first one refers to the classification of a
special class of finite dimensional Jordan algebras while the
second one to a possible further development of the
\emph{invariant theory} for Jordan algebras.

\subsection*{Application: flag Jordan algebras}

This section deals with the following special case of extending structures:

\begin{definition} \delabel{flagex}
Let $A$ be a Jordan algebra and $E$ a vector space containing $A$
as a subspace. A Jordan algebra structure on $E$ such that $A$ is
a subalgebra is called a \emph{flag extending structure} of $A$ to
$E$ if there exists a finite chain of subalgebras of $E$
\begin{equation} \eqlabel{lant}
E_0 := A  \subset E_1 \subset \cdots \subset E_m = E
\end{equation}
such that $E_i$ has codimension $1$ in $E_{i+1}$, for all $i = 0,
\cdots, m-1$. A finite dimensional Jordan algebra $E$ is called
\emph{flag} if it is a flag extending structure of $\{0\}$.
\end{definition}

The flag extending structures of $A$ to $E$ can be obtained by following the method described below. We start by describing and classifying all unified products $A \,\natural \, V_1$, for a $1$-dimensional vector space $V_1$. The next natural step consists in describing and classifying all unified products between the unified products previously described and a $1$-dimensional vector space. Continuing this process will render all flag extending structures of $A$ to $E$ after ${\rm dim}_k
(V)$ steps.

The following concept will be useful for our approach:

\begin{definition}\delabel{lambdaderivariii}
Let $A$ be a Jordan algebra and $V$ a vector space of dimension
$1$ with basis $\{x\}$. A \emph{flag datum} of $A$ is a system $(D, \,
\lambda, \,  a_0, \, \alpha_0) \in {\rm End}_k (A) \times A^*
\times A \times k$ such that $A_{(D, \, \lambda, \,  a_0, \, \alpha_0)} := A \,\natural \, kx$
with the multiplication given for all $a$, $b\in A$ by:
\begin{equation}\eqlabel{extenddim2022}
(a, x) \circ (b, x) = \bigl(a\cdot b + D(a + b) + a_0, \,\,\,
(\lambda(a + b) + \alpha_0) \, x \, \bigl)
\end{equation}
is a Jordan algebra. The set of all flag data of $A$ will be denoted by ${\mathcal F} (A)$.
\end{definition}

Explicitly, if $\{e_i \, | \, i\in I\}$ is a $k$-basis of a Jordan algebra $A$ then
$A_{(D, \, \lambda, \, a_0, \, \alpha_0)}$ is the commutative algebra
having $\{x, \, e_i \, | \, i\in I\}$ as a $k$-basis and the
multiplication given for all $i$, $j\in I$ by:
\begin{equation}\eqlabel{extenddim20}
e_i \circ e_j := e_i \cdot e_j, \quad e_i \circ x := x \circ e_i := D(e_i)
+ \lambda(e_i) \, x, \quad x \circ x := a_0 + \alpha_0 \, x
\end{equation}

The set of all Jordan extending structures ${\mathcal J} {\mathcal
E} {\mathcal S} \, (A, \, V)$ of a Jordan algebra $A$ through a
$1$-dimensional vector space $V$ is parameterized by ${\mathcal F}
(A)$.

\begin{proposition}\prlabel{unifdim1}
Let $A$ be a Jordan algebra and $V$ a vector space of dimension
$1$ with basis $\{x\}$. Then there exists a bijection between the
set ${\mathcal J} {\mathcal E} {\mathcal S} \, (A, \, V)$ of all
Jordan extending structures of $A$ through $V$ and the set
${\mathcal F} (A)$ of all flag data of $A$. Through the above
bijection, the Jordan extending structure $\Omega(A, \, V)  =
\bigl(\triangleleft, \, \triangleright, f, \cdot_V \bigl)$
corresponding to $(D, \, \lambda, \,  a_0, \, \alpha_0) \in
{\mathcal F} (A)$ is given for all $a \in A$ by:
\begin{equation}\eqlabel{extenddim1}
x \triangleleft a = \lambda (a) x, \quad x \triangleright a =
D(a), \quad f (x, x) = a_0, \quad x \cdot_V \, x = \alpha_0 \, x
\end{equation}
\end{proposition}

\begin{proof}
Since $V := kx$ has dimension $1$, the set of all bilinear maps
$\triangleleft : V \times A \to V$, $\triangleright : V \times A
\to A$, $f: V\times V \to A$ and $\cdot_V : V\times V \to V$ is in
bijection with the set of all systems $(D, \, \lambda, \, a_0, \,
\alpha_0) \in {\rm End}_k (A) \times A^* \times A \times k$ and
the bijection is given such that \equref{extenddim1} holds. The
proof is now finised using \thref{1} once we observe that the multiplication given
by \equref{extenddim2022} coincides with the one of the unified product
\equref{brackunif} associated to the extending datum defined by \equref{extenddim1}.
\end{proof}

\begin{remark}\relabel{unifbicross}
The compatibility conditions of a flag datum $(D, \, \lambda, \,  a_0, \, \alpha_0) \in
{\mathcal F} (A)$ of a Jordan algebra $A$ as defined in \deref{lambdaderivariii} can be written down explicitly using \thref{1} and \prref{unifdim1} by a straightforward computation. For instance, the axioms (E1)-(E7) from \thref{1}
are equivalent to the following five compatibility conditions:
\begin{eqnarray*}
&& a\cdot D(a^2) + \lambda(a^2) \, D(a) = a^2 \cdot D(a) + \lambda(a) D(a^2), \\
&& D(a_0 \cdot a) = a_0 \cdot D(a) + \lambda(a) D(a_0), \quad D^2 (a_0) + \lambda(a_0) \, a_0 = a_0^2 + \alpha_0 \,  D(a_0), \\
&& \lambda (a_0 \cdot a) = \lambda (a_0) \, \lambda(a), \quad \lambda(D(a_0)) = 0
\end{eqnarray*}
for all $a\in A$. Two more complicated compatibilities which we do not write down arise from
the missing relations \equref{missingrel}.
\end{remark}

Using \thref{classif} and \prref{unifdim1} we obtain:

\begin{corollary}\colabel{balam}
Let $A$ be a Jordan algebra. Then a Jordan algebra $E$ contains
$A$ as a subalgebra of codimension $1$ if and only if there exists
$(D, \, \lambda, \,  a_0, \, \alpha_0) \in {\mathcal F} (A)$ a
flag datum of $A$ such that $E \cong A_{(D, \, \lambda, \,  a_0,
\, \alpha_0)}$.
\end{corollary}

Applying the equivalence relation in \deref{echiaa} for ${\mathcal F} (A)$ comes down to the following:

\begin{definition} \delabel{echivtwderivari}
Two flag data $(D, \, \lambda, \,  a_0, \, \alpha_0)$ and $(D', \,
\lambda', \,  a_0', \, \alpha_0') \in {\mathcal F} (A)$ of a
Jordan algebra $A$ are called \emph{equivalent} and we denote this
by $(D, \, \lambda, \,  a_0, \, \alpha_0) \equiv (D', \, \lambda',
\, a_0', \, \alpha_0')$ if $\lambda = \lambda'$ and there exists a
pair $(r, \, u) \in A \times k^*$ such that for all $a \in A$ we
have:
\begin{eqnarray}
D(a) &=& u \, D'(a) + a \cdot r - \lambda'(a) \, r \eqlabel{lzeci3a} \\
\alpha_0 &=& u \, \alpha_0' + 2 \lambda'(r) \eqlabel{lzeci2b} \\
a_0 &=& u^2 \, a'_0 + r \cdot r + 2 u \, D' (r) - u \, \alpha_0'
\, r - 2 \lambda'(r) \, r \eqlabel{lzeci1c}
\end{eqnarray}
\end{definition}

We can now provide the classification of all Jordan algebras $A_{(D, \, \lambda, \, a_0,
\, \alpha_0)}$. This will be done by computing ${\mathcal
J}_{A} (V, \, A)$ for a $1$-dimensional vector space $V$ and constitutes the main step in classifying flag Jordan algebras.

The following can be obtained as a consequence of all the above together:

\begin{theorem}\thlabel{clasdim1}
Let $A$ be a Jordan algebra of codimension $1$ in the vector space
$E$ and $V$ a complement of $A$ in $E$. Then, $\equiv$ from
\deref{echivtwderivari} is an equivalence relation on the set
${\mathcal F} (A)$ of all flag data of $A$ and
$$
{\rm Jext} \, (E, \, A) \cong {\mathcal J}_{A} (V, \, A) \cong
{\mathcal F} (A)/\equiv
$$
The bijection between ${\mathcal F} (A)/\equiv$ and ${\rm Jext} \,
(E, \, A)$ is given by:
$$
\overline{ (D, \, \lambda, \,  a_0, \, \alpha_0) } \mapsto A_{(D,
\, \lambda, \, a_0, \, \alpha_0)}
$$
where $\overline{ (D, \, \lambda, \,  a_0, \, \alpha_0) }$ is the
equivalence class of $(D, \, \lambda, \,  a_0, \, \alpha_0)$ via
$\equiv$ and $A_{(D, \, \lambda, \, a_0, \, \alpha_0)}$ is the
Jordan algebra constructed in \equref{extenddim2022}.
\end{theorem}

The next example computes ${\mathcal J}_{A} (k, \, A)$
and then describes all Jordan algebra structures which extend the
Jordan algebra structure from $A$ to a vector space of dimension
$1 + {\rm dim}_k (A)$. The detailed computations are rather long
but straightforward and will be omitted.

\begin{example}
Let $A := k^n$ be the abelian Jordan algebra of dimension $n$,
i.e. $a \cdot b = 0$, for all $a$, $b\in k^n$. We define the
following sets:
\begin{eqnarray*}
{\mathcal F}_1 (k^n) &:=& \{(D, \, a_0, \, \alpha_0) \in {\rm
End}_k (k^n) \times k^n \times k \, | \,\, D^2 (a_0) = \alpha_0 \,
D(a_0) \}\\
{\mathcal F}_2 (k^n) &:=& \{(D, \, \lambda, \, a_0, \, \alpha_0)
\in {\rm End}_k (k^n) \times (k^n)^* \times k^n \times k \, | \,\,
D (a_0) = \lambda (a_0) = 0, \,\, \lambda \neq 0 \}
\end{eqnarray*}
Then the classifying object ${\mathcal J}_{k^n} (k, \, k^n)$ is
the coproduct of the following sets:
$$
{\mathcal J}_{k^n} (k, \, k^n) \cong \bigl({\mathcal F}_1
(k^n)/\equiv_1 \bigl) \, \sqcup \, \bigl( {\mathcal F}_1
(k^n)/\equiv_2 \bigl)
$$
where $\equiv_1$ and $\equiv_2$ are the following equivalence
relations: $(D, \, a_0, \, \alpha_0) \equiv_1 (D', \, a_0', \,
\alpha_0')$ if and only if there exists a pair $(r, \, u) \in k^n
\times k^*$ such that for all $a\in k^n$ we have:
$$
D(a) = u D'(a), \quad \alpha_0 = u \alpha_0', \quad a_0 = u^2 \,
a'_0 + 2 u \, D' (r) - u \, \alpha_0' \, r
$$
and $(D, \, \lambda, \, a_0, \, \alpha_0) \equiv_2 (D', \,
\lambda', \, a_0', \, \alpha_0')$  if and only if $\lambda =
\lambda'$ and there exists a pair $(r, \, u) \in k^n \times k^*$
such that for all $a\in k^n$ we have:
\begin{eqnarray*}
D(a) &=& u \, D'(a) - \lambda'(a) \, r , \quad  \alpha_0 = u \, \alpha_0' + 2 \lambda'(r)   \\
a_0 &=& u^2 \, a'_0 + 2 u \, D' (r) - u \, \alpha_0' \, r - 2
\lambda'(r) \, r
\end{eqnarray*}
Any $(n+1)$-dimensional Jordan algebra having a $n$-dimensional
abelian subalgebra is isomorphic to $k^n_{(D, \, a_0, \,
\alpha_0)}$, for some $(D, \, a_0, \, \alpha_0) \in {\mathcal F}_1
(k^n)$ or to $k^n_{(D, \, \lambda, \, a_0, \, \alpha_0)}$, for
some $(D, \, \lambda, \,  a_0, \, \alpha_0) \in {\mathcal F}_2
(k^n)$.
\end{example}

\subsection*{Application: an Artin type theorem for Jordan algebras}
Let $A$ be a Jordan algebra, $V$ a vector space and $\Omega(A, \,
V) = \bigl(\triangleleft, \, \triangleright, \, f, \, \cdot_V
\bigl)$ a Jordan extending structure of $A$ through $V$ such that
$\triangleright$ is trivial, i.e. $x \triangleright a = 0$, for
all $x\in V$ and $a \in A$. Then, in this case the compatibility
conditions (E1)-(E8) of \thref{1} take a simpler form and
the associated unified product, called the
\emph{twisted product} of $A$ and $V$,
will be denoted by $A \, \#^{\bullet} \, V $. Thus $A \, \#^{\bullet} \,
V = A \times \, V $ with the multiplication given by:
\begin{equation*}
(a, x) \circ (b, y) := \bigl(a \cdot b + f(x, y), \,\, x
\triangleleft b + y\triangleleft a + x\cdot_V y \bigl)
\end{equation*}
for all $a$, $b \in A$ and $x$, $y \in V$. Twisted products are
generalizations of spin factor Jordan algebras from
\exref{spinfactorm} which can be recovered by setting $A := k$ and $\cdot_V
:= 0$. These products will be the main ingredient in proving the following Artin type theorem. As in the classical case, the Jordan algebra version of the Artin theorem will provide a way of reconstructing a Jordan algebra out of its subalgebra of invariants.

\begin{theorem} \thlabel{recsiGal}
Let $G$ be a finite group whose order is invertible in $k$ and $A$ a Jordan algebra. Assume the group morphism $\varphi: G \to
{\rm Aut}_{\rm J} (A)$ is an action of $G$ on a $A$. If $A^G := \{a \in A \, | \, \varphi (g) (a) = a, \forall g\in G \} \subseteq A$ is the corresponding
subalgebra of invariants and $V$ a complement of $A^G$ in $A$, then there exists a Jordan extending structure of $A^G$ through
$V$ and an isomorphism of Jordan algebras $A \cong A^G \,
\#^{\bullet} \, V$.
\end{theorem}
\begin{proof}
We denote $\varphi (g) (a) = g \rightharpoonup a$,
for all $g\in G$ and $a\in A$. First, note that $A^G$ is a subalgebra of $A$ and $g
\rightharpoonup (a \cdot b) = (g\rightharpoonup a) \cdot
(g\rightharpoonup b)$, for all $g\in G$, $a$, $b\in A$. We define
the trace map as follows for all $x\in A$:
\begin{equation}\eqlabel{tracedef}
t : A \to A^G, \quad t (x) :=  |G|^{-1} \, \sum_{g\in G} \, g
\rightharpoonup x
\end{equation}
Observe that $t(x) \in A^G$ and therefore for all $a \in A^G$ we have:
$$
t ( x \cdot a  ) = |G|^{-1} \, \sum_{g\in G} \, g \rightharpoonup
(x \cdot a) = |G|^{-1} \, \sum_{g\in G} \, (g\rightharpoonup x)
\cdot (g\rightharpoonup a) = t(x) \cdot a
$$
Secondly, we note that the trace map $t : A \to A^G$ is a linear
retraction of the canonical inclusion $A^G \hookrightarrow A$,
i.e. $t (a) = a$, for all $a\in A^G$. Now, if we compute the
canonical extending structure of $A^G$ through $V := {\rm Ker}
(t)$ associated to the trace map $t$, using the formulas
\equref{bala1}-\equref{bala4} from the proof of \thref{classif},
we obtain that for all $x \in V$ and $a \in A^G$ we have:
$$
x\triangleright_t a = t(x \cdot a) = t(x) \cdot a = 0,
$$
i.e. the action $\triangleright_t$ is trivial. Applying
once again \thref{classif} we obtain that the map defined for all
$a\in A^G$ and $x\in V$ by:
\begin{equation*}
\vartheta: A^G \, \#^{\bullet} \, V \to A, \qquad \vartheta(a, x)
:= a + x
\end{equation*}
is an isomorphism of Jordan algebras. This finishes the proof.
\end{proof}

\begin{remark} \relabel{hilbert}
If the group $G$ in \thref{recsiGal} is a finite cyclic group generated by an element $g$ then it can be easily seen that ${\rm Ker} (t) = \{ a - (g\rightharpoonup a)
\, | \, a \in A\}$.
\end{remark}

An invariant theory, similar to the classical theory of groups acting on associative algebras, can be developed for Jordan algebras by studying the Jordan algebra extension $A^G \subseteq A$. For instance, the following problems arise naturally in this context:

\emph{Let $\varphi: G \to {\rm Aut}_{\rm J} (A)$ be an action of a
group $G$ on a Jordan algebra $A$. If the subalgebra of invariants
$A^G$ is special (resp. (semi)simple, solvable, nilpotent, etc.),
is $A$ also special (resp. (semi)simple, solvable,
nilpotent, etc.)?}

\section{Crossed products and the extension problem for Jordan algebras} \selabel{Hilbertext}

In this section we deal with crossed products of Jordan algebras. As mentioned before, this important construction can be obtained as a special case of the unified product.
Indeed, let $\Omega(A, \, V) = \bigl(\triangleleft, \, \triangleright, \,
f, \{-, \, -\} \bigl)$ be an extending datum of the Jordan algebra
$A = (A, \cdot) $ through a vector space $V$ such that
$\triangleleft$ is trivial, i.e. $x \triangleleft a = 0$, for all
$x\in V$ and $a \in A$. Applying \thref{1} we obtain that
$\Omega(A, \, V) = \bigl(\triangleright, \, f, \, \cdot_V \bigl)$
is a Jordan extending structure of $A$ through $V$ if and only if
$(V, \cdot_V)$ is a Jordan algebra and the following
compatibilities hold for all $a\in A$ and $x$, $y$, $z\in V$:
\begin{enumerate}
\item[(CP1)] $f\colon V\times V \to A$ is a symmetric map;\\
\item[(CP2)] $a \cdot (x \triangleright a^2)  = a^2 \cdot (x
\triangleright a)$; \\
\item[(CP3)] $x \triangleright \bigl( f(x, x) \cdot a \bigl) + \,
x \triangleright (x^2 \triangleright a)  = f(x, x) \cdot (x
\triangleright a) + x^2 \triangleright (x \triangleright a)$; \\
\item[(CP4)] $x \triangleright \bigl(y \triangleright f(x, x)
\bigl) + \, x \triangleright f(x^2, y) + f\bigl( x^2 \,\cdot_V \,
y, \, x \bigl) \, \, = f(x, x) \cdot f(x, y) + x^2 \triangleright
f(x, y) + (x\cdot_ V \, y) \triangleright f(x, x) + f(x^2, \,
x\cdot_V \, y)$; \\
\item[(CP5)] The missing relations \equref{missingrel} hold for
the trivial action $\triangleleft$ .
\end{enumerate}

A system $(A, V, \triangleright, f)$ consisting of two Jordan
algebras $A = (A, \cdot)$, $V = (V, \cdot_V)$ and two bilinear
maps $\triangleright : V \times A\to A$, $f: V\times V \to A$
satisfying the five compatibility conditions (CP1)-(CP5) is called
a \emph{crossed system} of $A$ and $V$ and the associated unified
product, denoted by $A
\#_{\triangleright}^f \, V$, will be called the \emph{crossed product}
of the Jordan algebras $A$ and $V$. Thus, $A \#_{\triangleright}^f
\, V = A \times \, V $ with the multiplication given for any $a$,
$b \in A$ and $x$, $y \in V$ by:
\begin{equation}\eqlabel{brackcrosspr}
(a, x) \circ (b, y) := \bigl( a \cdot b + x \triangleright b +
y\triangleright a + f(x, y), \,\,  x \cdot_V \, y \bigl)
\end{equation}

If $(A, V, \triangleright, f)$ is a crossed system of two Jordan
algebras then, $A \cong A \times \{0\}$ is an ideal in $A
\#_{\triangleright}^f \, V$ since $(a, 0) \circ (b, y) := \bigl( a
\cdot b + y\triangleright a , \, 0 \bigl)$. Conversely, we can show that any Jordan algebra structure on a vector space
$E$ which contains a given Jordan algebra $A$ as an ideal is isomorphic to a crossed product.

\begin{corollary}\colabel{croslieide}
Consider $A$ to be a Jordan algebra and $E$ a vector space such that $A \subseteq E$. Then, any Jordan algebra structure on $E$ such that
$A$ becomes a Jordan ideal in $E$ is isomorphic to a crossed product $A \#_{\triangleright}^f \, V$ of Jordan
algebras.
\end{corollary}

\begin{proof}
Let $\cdot_E$ be a Jordan algebra structure on $E$ such that $A$
is an ideal in $E$. In particular, $A$ is a subalgebra of $E$ and
hence we can apply \thref{classif}. In this case the action
$\triangleleft = \triangleleft_p$ of the Jordan extending
structure $\Omega(A, \, V) = \bigl(\triangleleft_p, \,
\triangleright_p, \, f_p, (\cdot_V)_p \bigl)$ defined by
\equref{bala2} in the proof of \thref{classif} is trivial since
for any $x \in V$ and $a \in A$ we have $x \cdot_E a \in A$ and
hence $p (x \cdot_E a) = x \cdot_E a$. Thus, $x \triangleleft_p a
= 0$, i.e. the unified product $A \,\natural \,_{\Omega(A, V)} V =
A \#_{\triangleright}^f \, V $ is the crossed product of the
Jordan algebras $A$ and $V:= {\rm Ker}(p)$.
\end{proof}

Using \coref{croslieide} we obtain the following result which shows
that crossed products play a crucial role in the classification
of all Jordan algebras of a given finite dimension.

\begin{corollary}\colabel{iterare}
Any finite dimensional Jordan algebra is isomorphic to an
iteration of crossed products of the form $\Bigl( \cdots \bigl(
(S_1 \# \, S_2)\# \, S_3 \bigl) \# \, \cdots \# \, S_t\Bigl)$,
where $S_i$ is a finite dimensional simple Jordan algebra for all
$i= 1,\cdots, t$.
\end{corollary}

\begin{proof}
Let $A$ be a Jordan algebra of dimension $n$. We prove the statement by
induction on $n$. If $n = 1$ then $A \cong k $ is a simple
algebra. If $n > 1$ then we distinguish two cases. Firstly, if $A$ is simple there is
nothing to prove. On the contrary, if $A$ has a proper ideal
$\{0\} \neq I \neq A$, it follows from \coref{croslieide} that $A
\cong I \# \, (A/I)$, a crossed product of $I$ and $A/I$. If $I$
and $A/I$ are simple the proof is finished; Otherwise, since ${\rm dim}_k (I)$, $ {\rm dim}_k (A/I) < n$, the conclusion follows by induction.
\end{proof}

\subsection*{Application: the extension problem}

We deal with the extension problem for Jordan algebras in the classical way, by using crossed products. Indeed, note first any crossed product $A \#_{\triangleright}^f \, V$ of two Jordan algebras $A$ and $V$ is an extension of $V$ by $A$ via the
following canonical exact sequence:
\begin{eqnarray} \eqlabel{extencrosb}
\xymatrix{ 0 \ar[r] & A \ar[r]^{i_A} & {A \#_{\triangleright}^f \,
V} \ar[r]^{\pi_V} & V \ar[r] & 0 }
\end{eqnarray}
where $i_A : A \to A \#_{\triangleright}^f \, V$, $i_A (a) := (a,
0)$ and $\pi_V : A \#_{\triangleright}^f \, V \to V$, $\pi_V (a, x)
:= x$, for all $a\in A$ and $x\in V$. Conversely, we have:

\begin{theorem} \thlabel{extenscros}
Let $A$ and $V$ be two Jordan algebras and $(E, \, i, \, \pi)$ an
extension of $V$ by $A$, that is there exists an exact sequence of
Jordan algebras
\begin{eqnarray*}
\xymatrix{ 0 \ar[r] & A \ar[r]^{i} & E  \ar[r]^{\pi} & V \ar[r] &
0 }
\end{eqnarray*}
Then $(E, \, i, \, \pi)$ is cohomologous, in the sense of
\deref{echivextedna}, to a crossed product extension of the form
\equref{extencrosb}.
\end{theorem}

\begin{proof}
We identify $A \cong {\rm Im} (i) = {\rm Ker} (\pi) \subseteq
E$ which allows us to assume that $A$ is an ideal of $E$. Since we are working
over a field, we can find a $k$-linear section $s : V \to E$ of
$\pi$, i.e. $\pi (s(x)) = x$, for all $x\in V$. Then $\psi : A
\times V \to E$, $\psi (a, x) := a + s(x)$ is an isomorphism of
vector spaces with the inverse $\psi^{-1} (y) = \bigl(y - s
(\pi(y)), \,\, \pi(y) \bigl)$, for all $y\in E$. Using the section
$s$ we can define the following two bilinear maps:
\begin{eqnarray*}
\triangleright &=& \triangleright_{s} \, \, : V \times A \to A,
\,\,\,\,\,
x \triangleright a := s(x) \cdot_E a \eqlabel{act2}\\
f &=& f_s \,\,\, : V \times V \to A, \,\,\,\, f(x, \, y) := s(x)
\cdot_E s(y) - s(x \cdot_V y) \eqlabel{coc1}
\end{eqnarray*}
for all $x$, $y \in V$ and $a \in A = {\rm Ker} (\pi)$. Since
$\pi$ is a Jordan algebra map and $s$ is a section of it we can
easily see that these are well-defined maps. We shall prove that
$(A, \, V, \, \triangleright_{s}, \, f_s)$ is a crossed system and
$\psi : A \#_{\triangleright_s}^{f_s} \,\, V \to E$ is an
isomorphism of Jordan algebras that stabilizes $A$ and co-stabilizes
$V$ and this will finish the proof. Instead of proving the compatibility
conditions (CP1)-(CP5) for the pair $(\triangleright_{s}, \, f_s)$ by a
rather long computation we will use the same trick as in \thref{classif} and \thref{1}.
First, the map $\psi: A \times V \to E$, $\psi(a, \, x) = a +
s(x)$ is a linear isomorphism between the Jordan algebra $E = (E,
\cdot_E)$ and the direct product of vector spaces $A \times V$.
Thus, there exists a unique Jordan algebra structure on $A \times
V$ such that $\psi$ is an isomorphism of Jordan algebras. This
unique multiplication $\bullet$ on $A \times V$ is given for all
$a$, $b \in A$ and $x$, $y\in V$ by:
\begin{eqnarray*}
(a, x) \bullet (b, y) &=& \psi^{-1} \bigl(\psi(a, x) \cdot_E
\psi(b, y)\bigl) = \psi^{-1} \bigl(  ( a + s(x) ) \cdot_E (b + s(y)) \bigl) \\
&=& \psi^{-1} \bigl( a \cdot_E b + a \cdot_E s(y) + s(x) \cdot_E b + s(x) \cdot_E s(y) \bigl) \\
&=& \bigl (a \cdot_A b + a \cdot_E s(y) + s(x) \cdot_E b + s(x)
\cdot_E \, s(y) - s(x\cdot_V y), \,\, x\cdot_V \, y \bigl) \\
&=& \bigl( a \cdot b + x \triangleright b + y\triangleright a +
f(x, y), \,\,  x \cdot_V \, y \bigl)
\end{eqnarray*}
This shows that the multiplication $\bullet$ on $A\times V$ coincides with the one given in
\equref{brackcrosspr}. Applying now \thref{1} for $\triangleleft
:= 0$ we obtain that $(A, \, V, \, \triangleright_{s}, \, f_s)$ is
a crossed system, $\psi : A \#_{\triangleright_s}^{f_s} \,\, V \to
E$ is an isomorphism of Jordan algebras and the diagram
\begin{eqnarray*}
\xymatrix {& A \ar[r]^{i_A} \ar[d]_{Id} & {A
\#_{\triangleright_s}^{f_s} \,\, V  }
\ar[r]^{\pi_V} \ar[d]^{\psi} & V \ar[d]^{Id}\\
& A \ar[r]^{i} & {E}\ar[r]^{\pi} & V}
\end{eqnarray*}
is commutative since $\pi (\psi (a, x)) = \pi (a) + \pi(s(x)) = x
= \pi_V (a, x)$, for all $a\in A$ and $x\in V$. The proof is now
finished.
\end{proof}

\thref{extenscros} shows that computing the classifying object ${\rm Ext} \, (V, \, A)$ reduces to the classification of all crossed product extensions of
the form \equref{extencrosb}.

Given two Jordan algebras $A$ and $V$, we denote by
${\mathcal C} {\mathcal P} \, (A, \, V)$ the set of all pairs
$(\triangleright, \, f)$ of bilinear maps $\triangleright: V\times
A \to A$, $f: V\times V \to A$ satisfying axioms (CP1)-(CP5).

\begin{definition}\delabel{nonex}
Let $A$ and $V$ be two Jordan algebras. Two pairs
$(\triangleright, \, f)$ and $(\triangleright', \, f') \in
{\mathcal C} {\mathcal P} \, (A, \, V)$ are called
\emph{cohomologous} and we denote this by $(\triangleright, \, f)
\approx (\triangleright', \, f')$ if there exists a linear map $r:
V \to A$ such that
\begin{eqnarray}
x \triangleright ' a &=& x \triangleright a - a \cdot r(x) \eqlabel{coho32}\\
f'(x, \, y) &=& f ( x, \, y) +  r (x  \cdot_V \, y) + r (x) \cdot
\, r(y) - x \triangleright r(y) - y \triangleright r(x)
\eqlabel{coho32b}
\end{eqnarray}
for all $a \in A$, $x$, $y \in V$.
\end{definition}

We can now provide a theoretical answer to the extension problem for
Jordan algebras:

\begin{theorem} \thlabel{scchjor}
Let $A$ and $V$ be two Jordan algebras. Then:

$(1)$ $\equiv$ as defined in \deref{nonex} is an equivalence
relation on the set ${\mathcal C} {\mathcal P} \, (A, \, V)$ of
all crossed systems of $A$ and $V$. We denote the quotient set by
${\rm H}^2_{\rm nab} \, (V, \, A ) :=  {\mathcal C} {\mathcal P}
(A, \, V)/ \equiv $ and we call it the \emph{non-abelian
cohomology} of $A$ and $V$.

$(2)$ There exists a bijection given by:
$$
{\rm H}^2_{\rm nab} \, (V, \, A ) \to {\rm Ext} \, (V, \, A),
\qquad \overline{(\triangleright, f)} \rightarrow \bigl(A
\#_{\triangleright}^f \, V, \, i_A, \pi_V \bigl)
$$
where $\overline{(\triangleright, f)}$ is the equivalence class of
$(\triangleright, f)$ via $\equiv$, $\bigl(A \#_{\triangleright}^f
\, V, \, i_A, \pi_V \bigl)$ is the crossed product extension given
by \equref{extencrosb} and ${\rm Ext} \, (V, \, A)$ is the
classifying object of all extensions of $V$ by $A$ as introduced
in \deref{echivextedn}.
\end{theorem}

\begin{proof}
The proof follows from \thref{extenscros} once we observe that
$(\triangleright, \, f) \approx (\triangleright', \, f')$ in the
sense of \deref{nonex} if and only if there exists an isomorphism
of Jordan algebras $\psi: A \#_{\triangleright}^f \, V \to A
\#_{\triangleright '}^{f'} \, V $ which stabilizes $A$ and
co-stabilizes $V$. We just mention that the conditions given by
\equref{coho32} and \equref{coho32b} are the only non trivial ones among the
compatibilities listed in \deref{echiaa} for the trivial right actions $\triangleleft$
and $\triangleleft'$. We also use the last statement of
\leref{morfismuni}. Therefore, $\approx$ is an equivalence
relation on the set ${\mathcal C} {\mathcal P} \, (A, \, V)$ of
all crossed system $A$ and $V$ and the conclusion follows.
\end{proof}

\begin{remark}
As in the case of the extension problem from group theory or Lie theory,
\thref{scchjor} takes a simplified form in the 'abelian' case
(i.e. when $a\cdot b := 0$, for all $a$, $b\in A$). These are called 'null
extensions' or 'central extensions' and were first considered in \cite[Theorem
12]{jacb}\footnote{In \cite{jacb} the cocyles $f: V\times V \to A$
are called 'factor sets'.}. Central extensions were also
studied in \cite[Section 3]{AbH} in connection to the classification of
nilpotent Jordan algebras.
\end{remark}

The next example parameterizes all extensions of a $1$-dimensional Jordan algebra through an arbitrary Jordan algebra $A$.

\begin{example}\exlabel{nonexcoho}
For a given Jordan algebra $A$ we will compute the non-abelian
cohomology ${\rm H}^2_{\rm nab} \, (k, \, A )$. Let $\{x\}$ be a basis in $k$: then
any Jordan algebra structure on $k$ has the form $x \cdot_k x =
\varepsilon \, x$. Up two an isomorphism there are two Jordan
algebra structures on $k$ corresponding to $\varepsilon \in \{0,
\, 1\}$. Fix $\varepsilon \in \{0, \, 1\}$ and denote by ${\rm
H}^2_{\rm nab} \, (k_{\varepsilon}, \, A )$ the corresponding
non-abelian cohomology. Then ${\mathcal C} {\mathcal P} (A, \, k_{\varepsilon})$ identifies with the set of all pairs $(D, \, a_0) \in {\rm End}_k
(A) \times A$ satisfying the following six compatibilities:
\begin{eqnarray*}
&& a \cdot D(a^2) = a^2 \cdot D(a), \, D(a_0 \cdot a) = a_0 \cdot
D(a), \, D^2(a_0) = a_0^2 + \varepsilon D(a_0), \, \varepsilon D(a_0) = D(a_0); \\
&& 2 \, a \cdot D^2 (a) + a\cdot D(a_0) + a_0 \cdot a + D(a) + D^2 (a^2) + 2\, D^3 (a) = \\
&& = a^2 \cdot a_0 + 2\, D(a)\cdot D(a) + 3 \, a_0 \cdot D(a) + D(a^2) + 3 \, D^2(a);\\
&& 2\, (D(a) \cdot b) \cdot a + (a_0 \cdot b) \cdot a + a \cdot D(b) + D(a^2 \cdot b) + 2\, D\bigl( D(a) \cdot b \bigl) = \\
&& = a^2 \cdot D(b) + 2\, D(a) \cdot  (b\cdot a) + 2\, D(a) \cdot D(b) + a_0 \cdot (b\cdot a) + D(b\cdot a)
\end{eqnarray*}
for all $a$, $b\in A$ and
$$
{\rm H}^2_{\rm nab} \, (k_{\varepsilon}, \, A ) \cong {\mathcal C}
{\mathcal P} (A, \, k_{\varepsilon} ) / \approx
$$
where $\approx$ is the following equivalence relation: $(D, \, a_0)
\approx (D', \, a_0')$ if and only if there exists $r \in A$ such
that
\begin{equation} \eqlabel{ultimaec}
D'(a) = D(a) - a \cdot r, \quad a_0' = a_0 + r^2 - 2 \, D(r) +
\varepsilon \, r
\end{equation}

for all $a\in A$. Indeed, the set of all pairs of bilinear maps $(\triangleright, f)$, where $\triangleright : k \times A \to A$, $f : k \times k \to A$ are in bijective
correspondence with the set of all pairs $(D, \, a_0) \in {\rm End}_k (A) \times A$ and the bijection is given by $x \triangleright a := D(a)$, for all $a\in A$ and
$f(x, x) := a_0$. Via this identification, through a laborious but straightforward computation, we can easily prove that $(\triangleright, f)$ is a crossed system if and only if the pair $(D, \, a_0)$ satisfies the above compatibility conditions: we just mention that the first three compatibilities are equivalent to the axioms (CP2)-(CP4) while the last three compatibilities are equivalent to (CP5) (i.e. the missing relation \equref{missingrel} written for
the trivial action $\triangleleft$). If $(D, \, a_0)$ is a pair as above then any extension
of $k_{\varepsilon}$ by the Jordan algebra $A$ is cohomologous with the crossed product
$A^{(D, \, a_0)} := A \#_{D}^{a_0} \, k$, which is the vector space $A \times k$ with the multiplication given for any $a$, $b\in A$ by:
$$
(a, x) \circ (b, x) := (a\cdot b + D(a+b) + a_0, \, \varepsilon x)
$$

As a special case, for a positive integer $n$, let us take $A := k^n_{0}$, the abelian $n$-dimensional Jordan algebra, and $\varepsilon := 0$.
Then, we can easily prove that
$$
{\rm H}^2_{\rm nab} \, (k_{0}, \, k^n_{0 } ) \cong \{D \in M_n (k) \,\,  | \,\, 2\, D^3 - 3 \, D^2 + D = 0  \}
$$
where $M_n (k)$ is the usual space of $n\times n$-matrices over $k$. Indeed, from the third and the fourth relation above we obtain $a_0 := 0$ while from the other compatibilities the only non-trivial one comes down to $D(a) + 2 D^2 (a) = 3 D^2 (a)$, for all $a \in A = k^n$. Moreover, in this case the equivalence relation \equref{ultimaec} becomes $D' = D$.
\end{example}

\end{document}